\documentclass[11pt]{article}
\usepackage{tikz}
\usepackage{amsfonts}
\usepackage{amsmath}
\usepackage{amsthm}
\usepackage{epic}
\usepackage{eepic}
\usepackage{color}
\usepackage{pst-node}

\newtheorem{thm}{Theorem}
\newtheorem{cor}{Corollary}

\newtheorem{obs}{Observation}

\newtheorem{lem}{Lemma}
\newtheorem{Problem}{Problem}
\theoremstyle{remark}

\newcommand{\pn}{{\rm pn}}
\newcommand{\barv}{\overline{v}}
\newcommand{\adom}{\gamma_{\rm a}}

\newcommand{\cF}{{\cal F}}
\newcommand{\cP}{{\cal P}}
\newcommand{\cAa}{{\cal A_{\gamma_{\rm a}}}}
\newcommand{\cAg}{{\cal A_{\gamma}}}

\newcommand{\1}{\vspace{0.04cm}}
\newcommand{\2}{\vspace{0.1cm}}

\let\oldenumerate\enumerate
\renewcommand{\enumerate}{
  \oldenumerate
  \setlength{\itemsep}{0.5pt}
  \setlength{\parskip}{0pt}
  \setlength{\parsep}{0pt}
}

\begin{document}

\title{On Accurate Domination in Graphs}
\author{$^{1}$Joanna Cyman, \, $^{2}$Michael A. Henning\thanks{Research
supported in part by the South African
National Research Foundation and the University of Johannesburg}, \, and \, $^{3}$Jerzy Topp\\
\\
$^1$Faculty of Applied Physics and Mathematics \\
Gda\'nsk University of Technology,  80-233 Gda\'nsk, Poland \\
\small {\tt Email: joanna.cyman@pg.edu.pl} \\
\\
$^2$Department of Pure and Applied Mathematics\\
University of Johannesburg \\
Auckland Park 2006, South Africa \\
\small {\tt Email: mahenning@uj.ac.za} \\
\\
$^3$Faculty of Mathematics, Physics and Informatics \\
University of Gda\'nsk, 80-952 Gda\'nsk, Poland \\
\small {\tt Email: jtopp@inf.ug.edu.pl}
}

\date{}
\maketitle

\begin{abstract}
A dominating set of a graph $G$ is a subset $D \subseteq V_G$ such that every vertex not in $D$ is adjacent to at least one vertex in $D$. The cardinality of a smallest dominating set of $G$, denoted by $\gamma(G)$, is the domination number of $G$. The accurate domination number of $G$, denoted by $\gamma_{\rm a}(G)$, is the cardinality of a smallest set $D$ that is a dominating set of $G$ and no $|D|$-element subset of $V_G \setminus D$ is a dominating set of $G$. We study graphs for which the accurate domination number is equal to the domination number. In particular, all trees $G$ for which $\gamma_{\rm a}(G) = \gamma(G)$ are characterized. Furthermore, we compare the accurate domination number with the domination number of different coronas of a graph.
\end{abstract}


\newpage
\section{Introduction and notation}

We generally follow the notation and terminology of \cite{ChartrandLesniakPing Zhang} and \cite{Haynes...Slater}. Let $G = (V_G,E_G)$ be a~graph with vertex set $V_G$ of order~$n(G) = |V_G|$ and edge set $E_G$ of size~$m(G) = |E_G|$. If $v$ is a~vertex of $G$, then the \emph{open neighborhood} of $v$ is the set $N_G(v)=\{u\in V_G\colon uv\in E_G\}$, while the \emph{closed neighborhood} of $v$ is the set $N_G[v]=N_G(v)\cup\{v\}$. For a subset $X$ of $V_G$ and a vertex $x$ in $X$, the set $\pn_G(x,X) = \{v \in V_G \mid N_G[v] \cap X = \{x\}\}$
is called the \emph{$X$-private neighborhood} of the vertex $x$, and it consists of those vertices of $N_G[x]$ which are not adjacent to any vertex in $X \setminus \{x\}$; that is, $\pn_G(x,X) = N_G[x] \setminus N_G[X\setminus\{x\}]$. The \emph{degree} $d_G(v)$ of a~vertex $v$ in~$G$ is the number of vertices in $N_G(v)$. A vertex of degree one is called a \emph{leaf} and its neighbor is called a \emph{support vertex}. The set of leaves of a graph $G$  is denoted by $L_G$, while the set of support vertices by $S_G$. For a set $S\subseteq V_G$, the subgraph induced by $S$ is denoted by $G[S]$, while the subgraph induced by $V_G \setminus S$ is denoted by $G-S$. Thus the graph $G - S$ is obtained from $G$ by deleting the vertices in $S$ and all edges incident with $S$.  Let $\kappa(G)$ denote the number of components of~$G$.

A \emph{dominating set} of a graph $G$ is a subset $D$ of $V_G$ such that every vertex not in $D$ is adjacent to at least one vertex in $D$, that is, $N_G(x)\cap D\ne \emptyset$ for every $x\in V_G \setminus D$. The \emph{domination number} of $G$, denoted by $\gamma(G)$, is the cardinality of a smallest dominating set of~$G$. An \emph{accurate dominating set} of~$G$  is a dominating set $D$ of $G$ such that no $|D|$-element subset of $V_G \setminus D$ is a~dominating set of $G$. The \emph{accurate domination number} of $G$, denoted by $\adom(G)$, is the cardinality of a smallest accurate dominating set of $G$. We call a dominating set of $G$ of cardinality $\gamma(G)$ a $\gamma$-\emph{set of $G$}, and an accurate dominating set of $G$ of cardinality $\adom(G)$ a $\adom$-\emph{set of $G$}. Since every accurate dominating set of $G$ is a dominating set of $G$, we note that $\gamma(G)\le \adom(G)$. The accurate domination in graphs was introduced by Kulli and Kattimani  \cite{KulliKattimani}, and further studied in a number of papers. A comprehensive survey of concepts and results on domination in graphs can be found in
\cite{Haynes...Slater}.

We denote the path and cycle on $n$ vertices by $P_n$ and $C_n$, respectively. We denote by $K_n$ the \emph{complete graph} on $n$ vertices, and by $K_{m,n}$ the \emph{complete bipartite graph} with partite sets of size~$m$ and $n$. The accurate domination numbers of some common graphs are given by the following formulas:

\begin{obs}
\label{formula}
The following holds.
\\[-26pt]
\begin{enumerate}
\item For $n \ge 1$, $\adom(K_n)= \lfloor \frac{n}{2} \rfloor + 1$ and $\adom(K_{n,n})= n + 1$. \1
\item For $n > m \ge 1$, $\adom(K_{m,n}) = m$.  \1
\item For $n \ge 3$, $\adom(C_n)= \lfloor \frac{n}{3} \rfloor - \lfloor \frac{3}{n} \rfloor+2$. \1
\item  For $n \ge 1$, $\adom(P_n)= \lceil \frac{n}{3} \rceil$ unless $n \in \{2,4\}$ when $\adom(P_n)= \lceil \frac{n}{3} \rceil + 1$
     {\rm (see Corollary~\ref{wniosek-sciezki})}.
\end{enumerate}
\end{obs}

In this paper we study graphs for which the accurate domination number is equal to the domination number. In particular, all trees $G$ for which $\adom(G)= \gamma(G)$ are characterized. Furthermore, we compare the accurate domination number with the domination number of different coronas of a graph. Throughout the paper, we use the symbol $\cAg(G)$ (respectively, $\cAa(G)$) to denote the set of all minimum dominating sets (respectively, minimum accurate dominating sets) of $G$.

\section{Graphs with  $\adom$ equal to $\gamma$}

We are interested in determining the structure of graphs for which the accurate
domination number is equal to the domination number. The question about such graphs has been stated in \cite{KulliKattimani}. We begin with the following general property of the graphs $G$ for which $\adom(G)= \gamma(G)$.

\begin{lem}
\label{twierdzenie1}
Let $G$ be a graph. Then $\adom(G)=\gamma(G)$ if and only if there exists a~set $D \in \cAg(G)$ such that $D \cap D' \ne  \emptyset$ for every set $D' \in \cAg(G)$.
\end{lem}
\begin{proof}
First assume that $\adom(G)=\gamma(G)$, and let $D$ be a minimum accurate dominating set of $G$. Since $D$ is a dominating set of $G$ and $|D|=\adom(G)=\gamma(G)$, we note that $D \in \cAg(G)$. Now let $D'$ be an arbitrary minimum dominating set of $G$. If $D \cap D' = \emptyset$, then $D' \subseteq V_G \setminus D$, implying that $D'$ would be a $|D|$-element dominating set of $G$, contradicting the fact that $D$ is an accurate dominating set of~$G$. Hence, $D \cap D' \ne \emptyset$.

Now assume that there exists a set $D \in \cAg(G)$ such that $D \cap D' \ne \emptyset$ for every set $D' \in \cAg(G)$. Then, $D$ is an accurate dominating set of $G$, implying that $\adom(G) \le |D| = \gamma(G) \le \adom(G)$. Consequently, we must have equality throughout this inequality chain, and so $\adom(G)=\gamma(G)$. \end{proof}

\medskip
It follows from Lemma \ref{twierdzenie1} that if $G$ is a disconnected graph, then
$\adom(G)=\gamma(G)$ if and only if $\adom(H)=\gamma(H)$ for at least one component $H$ of $G$. In particular, if $G$ has an isolated vertex, then
$\adom(G)=\gamma(G)$.  It also follows from Lemma \ref{twierdzenie1} that for a graph $G$, $\adom(G)=\gamma(G)$ if $G$ has one of the following properties: (1) $G$ has a unique minimum dominating set (see, for example, \cite{Fischermann} or \cite{GuntherHartnellMarkusRall} for some characterizations of such graphs); (2) $G$ has a vertex which belongs to every minimum dominating set of $G$ (see \cite{Mynhardt}); (3) $G$ has a vertex adjacent to at least two leaves. Consequently, there is no forbidden subgraph characterization for the class of graphs $G$ for which $\adom(G)=\gamma(G)$, as for any graph $H$, we can add an isolated vertex (or two leaves to one vertex of $H$), and in this way form a~graph $H'$ for which $\adom(H')= \gamma(H')$.

The \emph{corona} $F \circ K_1$ of a graph $F$ is the graph formed from $F$ by adding a new vertex $v'$ and edge $vv'$ for each vertex $v \in V(F)$. A graph $G$ is said to be a \emph{corona graph} if $G = F \circ K_1$ for some connected graph $F$. We note that each vertex of a corona graph $G$  is a leaf or it is adjacent to exactly one leaf of $G$. Recall that we denote the set of all leaves in a graph $G$ by $L_G$, and set of support vertices in $G$ by $S_G$.

\begin{lem} \label{twierdzenie-corona}
If $G$ is a corona graph, then
$\adom(G) > \gamma(G)$.
\end{lem}
\begin{proof} Assume that $G$ is a corona graph. If $G = K_1 \circ K_1$, then $G = K_2$ and $\adom(G) = 2$ and $\gamma(G) = 1$. Hence, we may assume that $G = F \circ K_1$ for some connected graph $F$ of order $n(F) \ge 2$. If $v \in V_G \setminus L_G$, then let $\barv$ denote the unique leaf-neighbor of $v$ in $G$. Now let $D$ be an arbitrary minimum dominating set of $G$, and so $D \in \cAg(G)$. Then, $|D \cap \{v, \barv \}| =1$ for every $v \in V_G \setminus L_G$. Consequently, $D$ and its complement $V_G \setminus D$ are minimum dominating sets of $G$. Thus, $D$ is not an accurate dominating set of $G$. This is true for every minimum dominating set of $G$, implying that $\adom(G) > \gamma(G)$.
\end{proof}

\begin{lem}
\label{l:support}
If $T$ is a tree of order at least three, then there exists a set $D \in \cAg(T)$ such that the following hold. \\[-27pt]
\begin{enumerate}
\item $S_T\subseteq D$.
\item $N_T(v)\subseteq V_T \setminus D$ or $|\pn_T(v,D)|\ge 2$ for every $v \in D \setminus S_T$.
\end{enumerate}
\end{lem}
\begin{proof} Let $T$ be a tree of order $n(T) \ge 3$. Among all minimum dominating sets of $T$, let  $D \in \cAg(T)$ be chosen that \2 \\
\indent (1) $D$ contains as many support vertices as possible. \\
\indent (2) Subject to (1), the number of components $\kappa(T[D])$ is as large as possible.

If the set $D$ contains a leaf~$v$ of $T$, then we can simply replace $v$ in $D$ with the support vertex adjacent to $v$ to produce a new minimum dominating set with more support vertices than $D$, a contradiction. Hence, the set $D$ contains no leaves, implying that $S_T \subseteq D$. Suppose, next, that there exists a vertex $v$ in $D$ that is not a support vertex of $T$ and such that $N_T(v) \not\subseteq V_T \setminus D$. Thus, $v$ has at least one neighbor in $D$; that is, $N_T(v)\cap D \ne \emptyset$. By the minimality of the set $D$, we therefore note that $\pn_T(v,D) \ne  \emptyset$. If $|\pn_T(v,D)| = 1$, say $\pn_T(v,D)=\{u\}$, then letting $D' = (D \setminus \{v\}) \cup \{u\}$, the set $D' \in \cAg(T)$ and satisfies $S_T \subseteq D \setminus \{v\} \subset D'$ and $\kappa(T[D']) > \kappa(T[D])$, which contradicts the choice of~$D$. Hence, if $v \in D$ is not a support vertex of $T$ and $N_T(v) \not\subseteq V_T \setminus D$, then $|\pn_T(v,D)|\ge 2$.
\end{proof}

We are now in a position to present the following equivalent characterizations of trees for which the accurate domination number is equal to the domination number.

\begin{thm}
\label{t:trees}
If $T$ is a tree of order at least two,
then the following statements are equivalent:
\\[-27pt]
\begin{enumerate}
\item[{\rm (1)}] $T$ is not a corona graph.
\item[{\rm (2)}] There exists a set $D \in \cAg(T)$ such that $\kappa(T-D) > |D|$.
\item[{\rm (3)}] $\adom(T)= \gamma(T)$.
\item[{\rm (4)}] There exists a set $D \in \cAg(T)$ such that $D \cap D' \ne \emptyset$ for every $D' \in \cAg(T)$.
\end{enumerate}
\end{thm}
\begin{proof} The statements (3) and (4) are equivalent by Lemma \ref{twierdzenie1}. The implication $(3)\Rightarrow (1)$ follows from Lemma \ref{twierdzenie-corona}. To prove the implication $(2)\Rightarrow (3)$, let us assume that $D \in \cAg(T)$ and $\kappa(T-D) > |D|$. Thus, $\gamma(T-D) \ge \kappa(T-D) > |D|= \gamma(T)$. This proves that $D$ is an accurate dominating set of $T$, and therefore $\adom(T)= \gamma(T)$.

Thus it suffices to prove that (1) implies (2). The proof is by induction on the order of a tree. The implication $(1) \Rightarrow (2)$ is obvious for trees of order two, three, and four. Thus assume that $T$ is a tree of order at least five and $T$ is not a corona graph. Let $D \in \cAg(T)$ and assume that $S_T \subseteq D$. Since $T$ is not a corona graph, the tree $T$ has a vertex which is neither a leaf nor adjacent to exactly one leaf. We consider two cases, depending on whether $d_T(v) \ge 3$ for some vertex $v \in S_T$ or $d_T(v)=2$ for every vertex $v \in S_T$.

\medskip
\emph{Case~1. $d_T(v) \ge 3$ for some $v \in S_T$.} Let $v'$ be a~leaf of $T$ adjacent to $v$. Let $T'$ be a~component of $T-\{v,v'\}$. Now let $T_1$ and $T_2$ be the subtrees of $T$ induced on the vertex sets $V_{T'} \cup \{v,v'\}$ and $V_T \setminus V_{T'}$, respectively. We note that both trees $T_1$ and $T_2$ have order strictly less than $n(T)$. Further, $V(T_1) \cap V(T_2) = \{v,v'\}$, $E(T_1) \cap E(T_2) = \{vv'\}$, and at least one of $T_1$ and $T_2$, say $T_1$, is not a corona graph. Applying the induction hypothesis to $T_1$, there exists a set $D_1 \in \cAg(T_1)$ such that $\kappa(T_1 - D_1) > |D_1|$. If $T_2$ is a corona graph, then choosing $D_2$ to be the set of support vertices in $T_2$ we note that $D_2 \in \cAg(T_2)$ and $\kappa(T_2 - D_2) = |D_2|$. If $T_2$ is not a corona graph, then applying the induction hypothesis to $T_2$, there exists a set $D_2 \in \cAg(T_2)$ such that $\kappa(T_2 - D_2) > |D_2|$. In both cases, there exists a set $D_2 \in \cAg(T_2)$ such that $\kappa(T_2 - D_2) \ge |D_2|$. We may assume that all support vertices of $T_1$ and $T_2$ are in $D_1$ and $D_2$, respectively. Thus, $v \in D_1 \cap D_2$, the union $D_1 \cup D_2$ is a $\gamma$-set of $T$, and $\kappa(T-(D_1\cup D_2))= \kappa(T_1-D_1) + \kappa(T_2-D_2) -1> |D_1|+|D_2|-1= |D_1\cup D_2|$.

\medskip
\emph{Case~2. $d_T(v)=2$ for every  $v\in S_T$.} We distinguish two subcases, depending on whether $D \setminus S_T \ne \emptyset$ or $D \setminus S_T = \emptyset$.

\medskip
\emph{Case~2.1. $D \setminus S_T \ne \emptyset$.} Let $v$ be an arbitrary vertex belonging to $D \setminus S_T$. It follows from the second part of Lemma~\ref{l:support} that there are two vertices $v_1$ and $v_2$ belonging to $N_T(v) \setminus D$. Let $R$ be the tree obtained from $T$ by adding a new vertex $v'$ and the edge $vv'$. We note that $D$ is a minimum dominating set of $R$ and $S_R \subseteq D$. Let $R'$ be the component of $R-\{v,v'\}$ containing $v_1$. Now let $R_1$ and $R_2$ be the subtrees of $R$ induced by the vertex sets $V_{R'} \cup \{v,v'\}$ and $V_R \setminus V_{R'}$, respectively. We note that both trees $R_1$ and $R_2$ have order strictly less than $n(T)$. Further, $V(R_1) \cap V(R_2) = \{v,v'\}$, $E(R_1) \cap E(R_2) = \{vv'\}$, and neither $R_1$ nor $R_2$ is a corona graph. By the induction hypothesis, there exists a set $D_1 \in \cAg(R_1)$ and a set $D_2 \in \cAg(R_2)$ such that $\kappa(R_1 - D_1) > |D_1|$ and $\kappa(R_2 - D_2) > |D_2|$. We may assume that all support vertices of $R_1$ and $R_2$ are in $D_1$ and $D_2$, respectively. Thus, $v \in D_1 \cap D_2$, the union $D_1 \cup D_2$ is a $\gamma$-set of $R$, and \[
\begin{array}{lcl}
\kappa(T-(D_1\cup D_2)) & = & \kappa(R-(D_1\cup D_2)) - 1 \\
& = & (\kappa(R_1-D_1) + \kappa(R_2-D_2) - 1) - 1 \\
& = & (\kappa(R_1-D_1) - |D_1| + \kappa(R_2-D_2) - |D_2| ) - 2 + |D_1| + |D_2| \\
& \ge & |D_1| + |D_2| \\
& = & |D_1 \cup D_2| + 1 \\
& > & |D_1\cup D_2|.
\end{array}
\]

\medskip
\emph{Case~2.2. $D \setminus S_T = \emptyset$.} In this case, we note that $D = S_T$. Let $v$ be an arbitrary vertex belonging to $D$ and assume that $N_T(v) = \{u, w\}$, where $u \in L_T$. If $w \in L_T$, then $T = K_{1,2}$, contradicting the assumption that $n(T) \ge 5$. If $w \in S_T$, then $T = P_4 = K_2 \circ K_1$, contradicting the assumption that $T$ is not a corona graph (and the assumption that $n(T) \ge 5$). Therefore, $w \in V_T \setminus (L_T \cup S_T)$. Thus,  $V_T \setminus (L_T\cup S_T)$ is nonempty and $T - D$ has $|D|$ one-element components induced by leaves of $T$ and at least one component induced by $V_T \setminus (L_T \cup S_T)$. Consequently, $\kappa(T-D)\ge |D|+1>|D|$. This completes the proof of Theorem~\ref{t:trees}.
\end{proof}

The equivalence of the statements (1) and (3) of Theorem~\ref{t:trees} shows that the trees $T$ for which $\adom(T)= \gamma(T)$ are easy to recognize. From Theorem~\ref{t:trees} and from the well-known fact that $\gamma(P_n) = \lceil n/3\rceil$ for every positive integer $n$, we also immediately have the following corollary which provides a slight improvement on Proposition~3 in~\cite{KulliKattimani}.

\begin{cor} \label{wniosek-sciezki}
For $n \ge 1$, $\adom(P_n)= \gamma(P_n)= \lceil n/3\rceil$ if and only if $n\in \mathbb{N} \setminus \{2,4\}$.
\end{cor}

\section{Domination of general coronas of a graph}

Let $G$ be a graph, and let $\cF =\{F_v\colon v\in V_G\}$ be a family of nonempty graphs indexed by the vertices of $G$. By $G \circ \cF $ we denote the graph with vertex set
\[
V_{G\circ \cF }= (V_G\times \{0\})\cup \bigcup_{v\in V_G}(\{v\}\times V_{F_v})
\]
and edge set determined by open neighborhoods defined in such a way that
\[
N_{G\circ \cF }((v,0)) = (N_G(v)\times \{0\})\cup (\{v\}\times V_{F_v})
\]
for every $v\in V_G$, and
\[
N_{G\circ \cF }((v,x)) = \{(v,0)\}\cup (\{v\}\times N_{F_v}(x))
\]
if $v\in V_G$ and $x\in V_{F_v}$. The graph $G\circ \cF $ is said to be the $\cF$-\emph{corona} of $G$. Informally,  $G \circ \cF$ is the graph obtained by taking a disjoint copy of $G$ and all the graphs of $\cF $  with additional edges joining each vertex $v$ of $G$ to every vertex in the copy of $F_v$. If all the graphs of the family $\cF $ are isomorphic to one and the same graph $F$ (as it was defined by Frucht and Harary \cite{FruchtHarary}), then we simply write $G\circ F$ instead of $G \circ \cF $. Recall that a graph $G$ is said to be a \emph{corona graph} if $G = F \circ K_1$ for some connected graph $F$.

The $2$-\emph{subdivided graph} $S_2(G)$ of a graph $G$ is the graph with vertex set
\[
V_{S_2(G)}= V_G \cup \bigcup_{vu \in E_G}\{(v,vu), (u,vu)\}
\]
and the adjacency is defined in such a way that
\[
N_{S_2(G)}(x)= \{(x,xy)\colon y\in N_G(x)\}
\]
if $x\in V_G\subseteq V_{S_2(G)}$, while
\[
N_{S_2(G)}((x,xy))= \{x\}\cup \{(y,xy)\}
\]
if $(x,xy)\in \bigcup_{vu\in E_G}\{(v,vu), (u,vu)\} \subseteq V_{S_2(G)}$. Less formally, $S_2(G)$ is the graph obtained from $G$ by subdividing every edge with two new vertices; that is, by replacing edges $vu$ of $G$ with disjoint paths $(v,(v,vu),(u,vu),u)$.

For a graph $G$ and a family $\cP = \{ \mathcal{P}(v) \colon v \in V_G\}$, where $\mathcal{P}(v)$  is a partition of the neighborhood $N_G(v)$ of the vertex $v$, by $G \circ \mathcal{P}$ we denote the graph with vertex set
\[
V_{G\circ \mathcal{P}} =(V_G\times \{1\}) \cup \bigcup_{v\in V_G} (\{v\} \times \mathcal{P}(v))
\]
and edge set
\[
E_{G\circ \mathcal{P}}= \bigcup_{v\in V_G} \{(v,1)(v,A)\colon A \in \mathcal{P}(v)\} \cup \bigcup_{uv\in E_G}\{(v,A)(u,B)\colon  (u\in A)  \wedge  (v\in B) \}.
\]

The graph $G\circ \mathcal{P}$ is called the $\mathcal{P}$-\emph{corona} of $G$ and was defined by Dettlaff et al. in \cite{Dettlaff...Zylinski}. It follows from this definition that if $\cP (v)= \{N_G(v)\}$ for every $v\in V_G$, then $G\circ \mathcal{P}$ is isomorphic to the corona $G\circ K_1$. On the other hand, if $\cP (v)= \{\{u\}\colon u\in N_G(v)\}$ for every $v\in V_G$, then $G\circ \mathcal{P}$ is isomorphic to the 2-subdivided graph $S_2(G)$ of~$G$. Examples of  $G\circ K_1$, $G\circ \mathcal{F}$, $G\circ \mathcal{P}$, and $S_2(G)$ are shown in Fig. \ref{rys1}.
In this case $G$ is the graph $(K_2\cup K_1)+K_1$ with vertex set $V_G=\{v, u, w, z\}$ and edge set $E_G= \{vu, vw, uw, wz\}$, where the family $\cF $ consists of the graphs $F_v=F_w=K_1$, $F_z=K_2$, and $F_u= K_2\cup K_1$, while $\cP = \{\cP (x)\colon x\in V_G\}$ is the family in which $\cP (v) =\{\{u, w\}\}$, $\cP (u) =\{\{v\}, \{w\}\}$, $\cP (w) =\{\{u, v\}, \{z\}\}$, and $\cP (z) =\{\{w\}\}$.

\begin{figure}[ht]
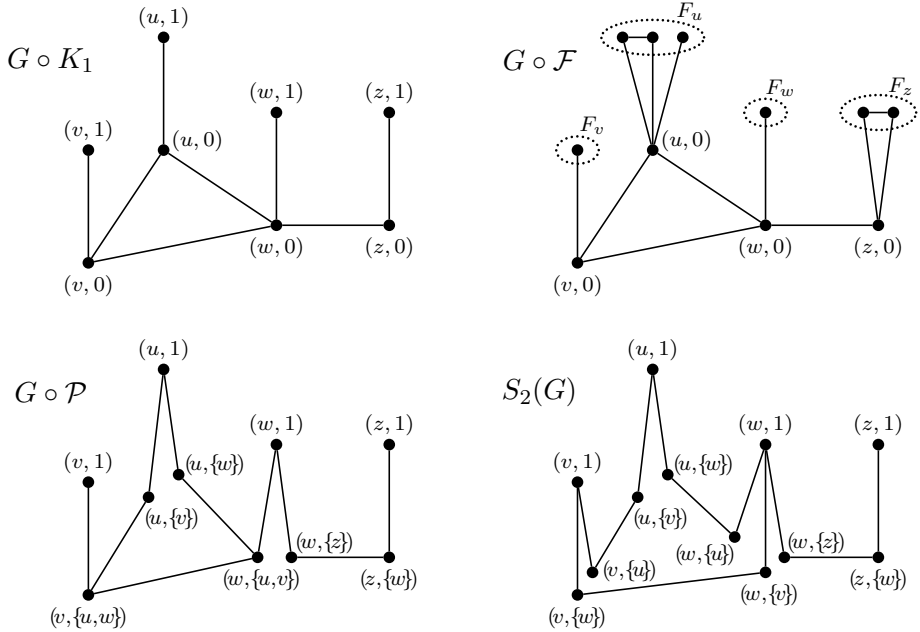

\vspace{6ex}

\rput(9.5,-1.7){\pspicture(6.5,5)
\cnode*[linewidth=0.5pt,fillstyle=solid,fillcolor=lightgray,linecolor=black](0,1){2.2pt}{s1}
\cnode*[linewidth=0.5pt,fillstyle=solid,fillcolor=lightgray,linecolor=black](2.5,1.5){2.2pt}{s2}
\cnode*[linewidth=0.5pt,fillstyle=solid,fillcolor=lightgray,linecolor=black](4,1.5){2.2pt}{s3}
\cnode*[linewidth=0.5pt,fillstyle=solid,fillcolor=lightgray,linecolor=black](1,2.5){2.2pt}{s4}
\cnode*[linewidth=0.5pt,fillstyle=solid,fillcolor=lightgray,linecolor=black](0,2.5){2.2pt}{k1}
\cnode*[linewidth=0.5pt,fillstyle=solid,fillcolor=lightgray,linecolor=black](2.5,3){2.2pt}{k2}
\cnode*[linewidth=0.5pt,fillstyle=solid,fillcolor=lightgray,linecolor=black](4,3){2.2pt}{k3}
\cnode*[linewidth=0.5pt,fillstyle=solid,fillcolor=lightgray,linecolor=black](1,4){2.2pt}{k4}

\ncline[linewidth=0.6pt, arrowsize=4pt 2]{-}{s1}{s2}
\ncline[linewidth=0.6pt, arrowsize=4pt 2]{-}{s2}{s3}
\ncline[linewidth=0.6pt, arrowsize=4pt 2]{-}{s1}{s4}
\ncline[linewidth=0.6pt, arrowsize=4pt 2]{-}{s4}{s2}
\ncline[linewidth=0.6pt, arrowsize=4pt 2]{-}{s1}{k1}
\ncline[linewidth=0.6pt, arrowsize=4pt 2]{-}{s2}{k2}
\ncline[linewidth=0.6pt, arrowsize=4pt 2]{-}{s3}{k3}
\ncline[linewidth=0.6pt, arrowsize=4pt 2]{-}{s4}{k4}

\rput(0,0.7){\scriptsize $(v,0)$}
\rput(2.5,1.2){\scriptsize $(w,0)$}
\rput(4,1.2){\scriptsize $(z,0)$}
\rput(1.45,2.63){\scriptsize $(u,0)$}
\rput(0,2.76){\scriptsize $(v,1)$}
\rput(2.5,3.26){\scriptsize $(w,1)$}
\rput(4,3.26){\scriptsize $(z,1)$}
\rput(1,4.26){\scriptsize $(u,1)$}
\rput(-0.5,3.7){$G\circ K_1$}

\endpspicture
\pspicture(8.5,5)
\cnode*[linewidth=0.5pt,fillstyle=solid,fillcolor=lightgray,linecolor=black](0,1){2.2pt}{s1}
\cnode*[linewidth=0.5pt,fillstyle=solid,fillcolor=lightgray,linecolor=black](2.5,1.5){2.2pt}{s2}
\cnode*[linewidth=0.5pt,fillstyle=solid,fillcolor=lightgray,linecolor=black](4,1.5){2.2pt}{s3}
\cnode*[linewidth=0.5pt,fillstyle=solid,fillcolor=lightgray,linecolor=black](1,2.5){2.2pt}{s4}
\cnode*[linewidth=0.5pt,fillstyle=solid,fillcolor=lightgray,linecolor=black](0,2.5){2.2pt}{k1}
\cnode*[linewidth=0.5pt,fillstyle=solid,fillcolor=lightgray,linecolor=black](2.5,3){2.2pt}{k2}
\cnode*[linewidth=0.5pt,fillstyle=solid,fillcolor=lightgray,linecolor=black](3.8,3){2.2pt}{k31}
\cnode*[linewidth=0.5pt,fillstyle=solid,fillcolor=lightgray,linecolor=black](4.2,3){2.2pt}{k32}
\cnode*[linewidth=0.5pt,fillstyle=solid,fillcolor=lightgray,linecolor=black](0.6,4){2.2pt}{k41}
\cnode*[linewidth=0.5pt,fillstyle=solid,fillcolor=lightgray,linecolor=black](1,4){2.2pt}{k42}
\cnode*[linewidth=0.5pt,fillstyle=solid,fillcolor=lightgray,linecolor=black](1.4,4){2.2pt}{k43}

\ncline[linewidth=0.6pt, arrowsize=4pt 2]{-}{s1}{s2}
\ncline[linewidth=0.6pt, arrowsize=4pt 2]{-}{s2}{s3}
\ncline[linewidth=0.6pt, arrowsize=4pt 2]{-}{s1}{s4}
\ncline[linewidth=0.6pt, arrowsize=4pt 2]{-}{s4}{s2}
\ncline[linewidth=0.6pt, arrowsize=4pt 2]{-}{s1}{k1}
\ncline[linewidth=0.6pt, arrowsize=4pt 2]{-}{s2}{k2}
\ncline[linewidth=0.6pt, arrowsize=4pt 2]{-}{s3}{k31}
\ncline[linewidth=0.6pt, arrowsize=4pt 2]{-}{k31}{k32}
\ncline[linewidth=0.6pt, arrowsize=4pt 2]{-}{s3}{k32}
\ncline[linewidth=0.6pt, arrowsize=4pt 2]{-}{s4}{k41}
\ncline[linewidth=0.6pt, arrowsize=4pt 2]{-}{s4}{k42}
\ncline[linewidth=0.6pt, arrowsize=4pt 2]{-}{s4}{k43}
\ncline[linewidth=0.6pt, arrowsize=4pt 2]{-}{k41}{k42}

\rput(0,0.7){\scriptsize $(v,0)$}
\rput(2.5,1.2){\scriptsize $(w,0)$}
\rput(4,1.2){\scriptsize $(z,0)$}
\rput(1.45,2.63){\scriptsize $(u,0)$}
\rput(0.2,2.83){\scriptsize $F_v$}
\rput(2.7,3.33){\scriptsize $F_w$}
\rput(1.5,4.35){\scriptsize $F_u$}
\rput(4.3,3.35){\scriptsize $F_z$}
\psellipse[linewidth=1pt,linecolor=black,linestyle=dotted,dotsep=1pt](1,4)(0.7,0.25)
\psellipse[linewidth=1pt,linecolor=black,linestyle=dotted,dotsep=1pt](4,3)(0.52,0.25)
\psellipse[linewidth=1pt,linecolor=black,linestyle=dotted,dotsep=1pt](0,2.5)(0.3,0.2)
\psellipse[linewidth=1pt,linecolor=black,linestyle=dotted,dotsep=1pt](2.5,3)(0.3,0.2)
\rput(-0.5,3.7){$G\circ \mathcal{F}$}
\endpspicture}
\vspace{28ex}

\rput(9.5,-1){\pspicture(6.5,5)
\cnode*[linewidth=0.5pt,fillstyle=solid,fillcolor=lightgray,linecolor=black](0,1){2.2pt}{s1}
\cnode*[linewidth=0.5pt,fillstyle=solid,fillcolor=lightgray,linecolor=black](2.25,1.5){2.2pt}{s21}
\cnode*[linewidth=0.5pt,fillstyle=solid,fillcolor=lightgray,linecolor=black](2.7,1.5){2.2pt}{s22}
\cnode*[linewidth=0.5pt,fillstyle=solid,fillcolor=lightgray,linecolor=black](4,1.5){2.2pt}{s3}
\cnode*[linewidth=0.5pt,fillstyle=solid,fillcolor=lightgray,linecolor=black](0.8,2.3){2.2pt}{s41}
\cnode*[linewidth=0.5pt,fillstyle=solid,fillcolor=lightgray,linecolor=black](1.2,2.6){2.2pt}{s42}
\cnode*[linewidth=0.5pt,fillstyle=solid,fillcolor=lightgray,linecolor=black](0,2.5){2.2pt}{k1}
\cnode*[linewidth=0.5pt,fillstyle=solid,fillcolor=lightgray,linecolor=black](2.5,3){2.2pt}{k2}
\cnode*[linewidth=0.5pt,fillstyle=solid,fillcolor=lightgray,linecolor=black](4,3){2.2pt}{k3}
\cnode*[linewidth=0.5pt,fillstyle=solid,fillcolor=lightgray,linecolor=black](1,4){2.2pt}{k4}

\ncline[linewidth=0.6pt, arrowsize=4pt 2]{-}{s1}{s21}
\ncline[linewidth=0.6pt, arrowsize=4pt 2]{-}{s22}{s3}
\ncline[linewidth=0.6pt, arrowsize=4pt 2]{-}{s1}{s41}
\ncline[linewidth=0.6pt, arrowsize=4pt 2]{-}{s42}{s21}
\ncline[linewidth=0.6pt, arrowsize=4pt 2]{-}{s1}{k1}
\ncline[linewidth=0.6pt, arrowsize=4pt 2]{-}{s21}{k2}
\ncline[linewidth=0.6pt, arrowsize=4pt 2]{-}{s22}{k2}
\ncline[linewidth=0.6pt, arrowsize=4pt 2]{-}{s3}{k3}
\ncline[linewidth=0.6pt, arrowsize=4pt 2]{-}{s41}{k4}
\ncline[linewidth=0.6pt, arrowsize=4pt 2]{-}{s42}{k4}
\rput(0,0.7){\scriptsize $(\!v,\!\{\!u,\!w\!\}\!)$}
\rput(2.3,1.2){\scriptsize $(\!w,\!\{\!u,\!v\!\}\!)$}
\rput(3.144,1.7){\scriptsize $(\!w,\!\{\!z\!\}\!)$}
\rput(4,1.2){\scriptsize $(\!z,\!\{\!w\!\}\!)$}
\rput(1.109,2.012){\scriptsize $(\!u,\!\{\!v\!\}\!)$}
\rput(1.68,2.73){\scriptsize $(\!u,\!\{\!w\!\}\!)$}
\rput(0,2.76){\scriptsize $(v,1)$}
\rput(2.5,3.26){\scriptsize $(w,1)$}
\rput(4,3.26){\scriptsize $(z,1)$}
\rput(1,4.26){\scriptsize $(u,1)$}
\rput(-0.5,3.7){$G\circ \mathcal{P}$}

\endpspicture
\pspicture(8.5,5)
\cnode*[linewidth=0.5pt,fillstyle=solid,fillcolor=lightgray,linecolor=black](0,1){2.2pt}{s11}
\cnode*[linewidth=0.5pt,fillstyle=solid,fillcolor=lightgray,linecolor=black](0.2,1.3){2.2pt}{s12}
\cnode*[linewidth=0.5pt,fillstyle=solid,fillcolor=lightgray,linecolor=black](2.09,1.77){2.2pt}{s21}
\cnode*[linewidth=0.5pt,fillstyle=solid,fillcolor=lightgray,linecolor=black](2.5,1.3){2.2pt}{s22}
\cnode*[linewidth=0.5pt,fillstyle=solid,fillcolor=lightgray,linecolor=black](2.75,1.5){2.2pt}{s23}
\cnode*[linewidth=0.5pt,fillstyle=solid,fillcolor=lightgray,linecolor=black](4,1.5){2.2pt}{s3}
\cnode*[linewidth=0.5pt,fillstyle=solid,fillcolor=lightgray,linecolor=black](0.8,2.3){2.2pt}{s41}
\cnode*[linewidth=0.5pt,fillstyle=solid,fillcolor=lightgray,linecolor=black](1.2,2.6){2.2pt}{s42}
\cnode*[linewidth=0.5pt,fillstyle=solid,fillcolor=lightgray,linecolor=black](0,2.5){2.2pt}{k1}
\cnode*[linewidth=0.5pt,fillstyle=solid,fillcolor=lightgray,linecolor=black](2.5,3){2.2pt}{k2}
\cnode*[linewidth=0.5pt,fillstyle=solid,fillcolor=lightgray,linecolor=black](4,3){2.2pt}{k3}
\cnode*[linewidth=0.5pt,fillstyle=solid,fillcolor=lightgray,linecolor=black](1,4){2.2pt}{k4}

\ncline[linewidth=0.6pt, arrowsize=4pt 2]{-}{s11}{s22}
\ncline[linewidth=0.6pt, arrowsize=4pt 2]{-}{s23}{s3}
\ncline[linewidth=0.6pt, arrowsize=4pt 2]{-}{s12}{s41}
\ncline[linewidth=0.6pt, arrowsize=4pt 2]{-}{s42}{s21}
\ncline[linewidth=0.6pt, arrowsize=4pt 2]{-}{s11}{k1}
\ncline[linewidth=0.6pt, arrowsize=4pt 2]{-}{s12}{k1}
\ncline[linewidth=0.6pt, arrowsize=4pt 2]{-}{s21}{k2}
\ncline[linewidth=0.6pt, arrowsize=4pt 2]{-}{s22}{k2}
\ncline[linewidth=0.6pt, arrowsize=4pt 2]{-}{s23}{k2}
\ncline[linewidth=0.6pt, arrowsize=4pt 2]{-}{s3}{k3}
\ncline[linewidth=0.6pt, arrowsize=4pt 2]{-}{s41}{k4}
\ncline[linewidth=0.6pt, arrowsize=4pt 2]{-}{s42}{k4}
\rput(0,0.7){\scriptsize $(\!v,\!\{\!w\!\}\!)$}
\rput(0.7,1.33){\scriptsize $(\!v,\!\{\!u\!\}\!)$}
\rput(2.55,1){\scriptsize $(\!w,\!\{\!v\!\}\!)$}
\rput(3.177,1.7){\scriptsize $(\!w,\!\{\!z\!\}\!)$}
\rput(1.7,1.53){\scriptsize $(\!w,\!\{\!u\!\}\!)$}
\rput(4,1.2){\scriptsize $(\!z,\!\{\!w\!\}\!)$}
\rput(1.109,2.012){\scriptsize $(\!u,\!\{\!v\!\}\!)$}
\rput(1.68,2.73){\scriptsize $(\!u,\!\{\!w\!\}\!)$}
\rput(0,2.76){\scriptsize $(v,1)$}
\rput(2.5,3.26){\scriptsize $(w,1)$}
\rput(4,3.26){\scriptsize $(z,1)$}
\rput(1,4.26){\scriptsize $(u,1)$}
\rput(-0.5,3.7){$S_2(G)$}
\endpspicture}

\vspace{18ex} \caption{Coronas of  $G=(K_2\cup K_1)+K_1$.}\label{rys1} \vspace{4ex}
\end{figure}

We now study relations between the domination number and the accurate domination number
of different coronas of a graph. Our first theorem specifies when these two numbers are equal for the $\cF$-corona $G\circ \cF $ of a graph $G$ and a~family $\cF $ of nonempty graphs indexed by the vertices of $G$.

\begin{thm}
\label{twierdzenie-corona-ogolnie}
If $G$ is a graph and $\cF = \{F_v \colon v\in V_G\}$ is a family of nonempty graphs indexed by the vertices of\, $G$, then the following holds. \\[-27pt]
\begin{enumerate}
\item[{\rm (1)}] $\gamma(G\circ \cF )= |V_G|$.
\item[{\rm (2)}] $\adom(G\circ \cF )= \gamma(G\circ \cF )$ if and only if $\gamma(F_v)>1$ for some vertex $v$ of $G$.
\item[{\rm (3)}] $|V_G|\le\adom(G\circ \cF )\le |V_G|+\min\{|V_{F_v}|\colon v\in V_G\}$.
\end{enumerate}
\end{thm}
\begin{proof} (1) It is obvious that $V_G\times \{0\}$ is a minimum dominating set of $G\circ \cF $ and therefore $\gamma(G\circ \cF )= |V_G\times \{0\}|=|V_G|$.

(2)  If $\gamma(F_v)>1$ for some vertex $v$ of $G$, then
\[
\gamma(G\circ \cF - (V_G\times \{0\}))= \sum_{v\in V_G}\gamma((G\circ \cF )[\{v\}\times V_{F_v}])= \sum_{v\in V_G}\gamma(F_v)>|V_G|= |V_G\times \{0\}|
\]
and this proves that no subset of $V_{G\circ \cF }-(V_G\times \{0\})$ of cardinality $|V_G\times \{0\}|$ is a dominating set of $G\circ \cF $. Consequently $V_G\times \{0\}$ is a minimum accurate dominating set of $G\circ \cF $ and therefore $\adom(G\circ \cF )=\gamma(G\circ \cF )$.

Assume now that $G$ and $\cF $ are such that $\adom(G\circ \cF )=\gamma(G\circ \cF )$. We claim that $\gamma(F_v)>1$ for some vertex $v$ of $G$. Suppose, contrary to our claim, that $\gamma(F_v)=1$ for every vertex $v$ of $G$. Then the set $U_v=\{x\in V_{F_v}\colon N_{F_v}[x]=V_{F_v}\}$, the set of universal vertices of $F_v$, is nonempty for every $v\in V_G$. Now, let $D$ be any minimum dominating set of $G\circ \cF $. Then, $|D|= \gamma(G\circ \cF ) = |V_G\times \{0\}|= |V_G|$, $|D\cap (\{(v,0)\} \cup (\{v\}\times U_v))|=1$, and the set $(\{(v,0)\} \cup (\{v\}\times U_v))-D$ is nonempty for every $v\in V_G$. Now, if $\overline{D}$ is a~system of representatives of the family $\{(\{(v,0)\} \cup (\{v\}\times U_v))-D\colon v\in V_G\}$, then $\overline{D}$ is a minimum dominating set of $G\circ \cF $. Since $\overline{D}$ and $D$ are disjoint, $D$ is not an accurate dominating set of $G\circ \cF $. Consequently, no minimum dominating set of $G\circ \cF $ is an accurate dominating set and therefore $\gamma(G\circ \cF ) <\adom(G\circ \cF )$, a~contradiction.

(3) The lower bound is obvious as $|V_G|=\gamma(G\circ \cF )\le\adom(G\circ \cF )$. Since $(V_G\times \{0\})\cup (\{v\}\times V_{F_v})$ is an accurate dominating set of $G\circ \mathcal{F}$ (for every $v\in V_G$), we also have the inequality $\adom(G\circ \cF )\le |V_G|+\min\{|V_{F_v}|\colon v\in V_G\}$. This completes the proof of Theorem~\ref{twierdzenie-corona-ogolnie}.
\end{proof}

As a consequence of Theorem~\ref{twierdzenie-corona-ogolnie}, we have the following result.

\begin{cor}
\label{twierdzenie-corona2}
If $G$ is a graph, then
$\adom(G\circ K_1)= \gamma(G\circ K_1)+1= |V_G|+1$. \end{cor}
\begin{proof} Since $\gamma(K_1)=1$, it follows from Theorem \ref{twierdzenie-corona-ogolnie} that  $\adom(G\circ K_1)\ge \gamma(G\circ K_1)+1= |V_G|+1$. On the other hand the set $(V_G\times \{0\})\cup \{(v,1)\}$ is an
accurate dominating set of $G\circ K_1$ and therefore $\adom(G\circ K_1)\le
|(V_G\times \{0\})\cup \{(v,1)\}| = |V_G|+1$. Consequently, $\adom(G\circ K_1)= \gamma(G\circ K_1)= |V_G|+1$. \end{proof}

From Theorem \ref{twierdzenie-corona-ogolnie} we know that $\adom(G\circ \cF )= \gamma(G\circ \cF )=|V_G|$ if and only if the family $\cF $ is such that $\gamma(F_v)>1$ for some $F_v\in \cF $, but we do not know any general formula for
$\adom(G\circ \cF )$ if $\gamma(F_v)=1$ for every $F_v\in \cF $.
The following theorem shows a formula for the domination number and general bounds for
the accurate domination number of a $\cP $-corona of a graph.

\begin{thm} \label{tw-general-corona}
If  $G$ is a graph and $\cP =\{\mathcal{P}(v) \colon v\in V_G\}$ is a family of partitions of the vertex neighborhoods of $G$, then the following holds. \\[-27pt]
\begin{enumerate}
\item[{\rm (1)}] $\gamma(G\circ \mathcal{P})=|V_G|$.
\item[{\rm (2)}] $\adom(G\circ \mathcal{P})\ge |V_G|$.
\item[{\rm (3)}] $\adom(G\circ \mathcal{P})\le |V_G|+\min\{\min\{|\cP (v)|\colon v\in V_G\}, 1+\min\{|A|\colon A\!\in \!\bigcup_{v\in V_G}\!\cP (v)\}\}$.
\end{enumerate}
\end{thm}
\begin{proof}  It follows from the definition of $G\circ \mathcal{P}$ that $V_G\times \{1\}$ is a dominating set of $G\circ \mathcal{P}$, and therefore $\gamma(G \circ \mathcal{P})\le |V_G\times \{1\}|=|V_G|$. On the other hand, let $D \in \cAg(G\circ \mathcal{P})$. Then $D \cap N_{G\circ \mathcal{P}}[(v,1)] \ne  \emptyset$ for every $v\in V_G$, and, since the sets $N_{G\circ \mathcal{P}}[(v,1)]$ form a partition of $V_{G\circ \mathcal{P}}$, we have
\[
\gamma(G\circ \mathcal{P})= |D|= |\bigcup_{v\in V_G} \left(D\cap N_{G\circ \mathcal{P}}[(v,1)]\right)|= \sum_{v\in V_G}|D \cap N_{G\circ \mathcal{P}}[(v,1)]| \ge |V_G|.
\]
Consequently, we have $|V_G|= \gamma( G\circ \cP )\le \adom(G\circ \mathcal{P})$, which proves (1) and (2).

From the definition of $G\circ \mathcal{P}$ it also follows that each of the sets $(V_G\times \{1\})\cup N_{G\circ \mathcal{P}}[(v,1)]$ (for every $v\in V_G$) and $(V_G\times \{1\})\cup N_{G\circ \mathcal{P}}[(v,A)]$ (for every $v\in V_G$ and $A\in \cP (v)$) is an accurate dominating set of $G\circ \mathcal{P}$. Hence,
\[
\begin{array}{lcl}
|(V_G\times \{1\})\cup N_{G\circ \mathcal{P}}[(v,1)]|
& = & |V_G\times \{1\}|+ |N_{G\circ \mathcal{P}}((v,1))| \\
& = & |V_G|+|\cP (v)| \\
& \ge & |V_G|+\min\{ |\cP (v)|\colon v\in V_G\} \\
& \ge & \adom(G\circ \mathcal{P}),
\end{array}
\]
and similarly
\[
\begin{array}{lcl}
|(V_G\times \{1\})\cup N_{G\circ \mathcal{P}}[(v,A)]|
& = &  |(V_G\times \{1\})\cup \{(v,1)\} \cup N_{G\circ \mathcal{P}}((v,A))| \\
& = & |V_G|+1+ |A| \\
& \ge & |V_G|+1+ \min\{|A|\colon A\in \bigcup_{v\in V_G}\cP (v)\}.
\end{array}
\]
Therefore,
\[
\adom(G\circ \mathcal{P})\le |V_G|+\min\{\min\{|\cP (v)|\colon v\in V_G\}, 1+\min\{|A|\colon A\in \bigcup_{v\in V_G}\cP (v)\}\}.
\]
This completes the proof of Theorem~\ref{tw-general-corona}. \end{proof}

We do not know all the pairs $(G,\cP )$ achieving equality in the upper bound for
the accurate domination number of a $\cP $-corona of a graph, but Theorem~\ref{tw-2-subdivision} and Corollaries~\ref{wniosek3} and~\ref{wniosek4} show that the bounds in Theorem \ref{tw-general-corona} are best possible. The next theorem also shows that the domination number and the accurate domination number of a $2$-subdivided graph are easy to compute.

\begin{thm} \label{tw-2-subdivision} If $G$ is a connected graph, then the following holds. \\[-27pt]
\begin{enumerate}
\item[{\rm (1)}] $\gamma(S_2(G))=|V_G|$. \2
\item[{\rm (2)}] $|V_G|\le \adom(S_2(G))\le |V_G|+2$. \2
\item[{\rm (3)}] $\adom(S_2(G)) = \left\{\begin{array}{rl} |V_G|+2,& \mbox{if \hspace{-0.3ex} $G$ is a cycle,}\\[1ex] |V_G|+1, & \mbox{if \hspace{-0.3ex} $G=K_2$,}\\[1ex] |V_G|,& \mbox{otherwise.}\end{array}\right.$
\end{enumerate}
\end{thm}
\begin{proof} The statement (1) follows from Theorem \ref{tw-general-corona}\,(1).

(2) The inequalities $|V_G|\le \adom(S_2(G))\le |V_G|+2$ are obvious if $G=K_1$. Thus assume that $G$ is a connected graph of order at least two. Let $u$ and $v$ be adjacent vertices of $G$. Then, $V_G\cup \{(v,vu),(u,vu)\}$ is an accurate dominating set of $S_2(G)$ and we have $|V_G|=\gamma(S_2(G))\le \adom(S_2(G))\le |V_G\cup \{(v,vu),(u,vu)\}|=|V_G|+2$.

(3) The connectivity of $G$ implies that there are three cases to consider.

\emph{Case 1. $|E_G|>|V_G|$.} In this case $S_2(G)-V_G$ has $|E_G|$ components and therefore no $|V_G|$-element subset of $V_{S_2(G)} \setminus V_G$ dominates $S_2(G)$. Hence, $V_G$ is an accurate dominating set of $S_2(G)$ and $\adom(S_2(G))= |V_G|$.

\emph{Case 2. $|E_G|=|V_G|$.} In this case, $G$ is a unicyclic graph. First, if $G$ is a cycle, say $G=C_n$, then $S_2(G)=C_{3n}$ and $\adom(S_2(G))= \adom(C_{3n}) =n+2=|V_G|+2$ (see Proposition 3 in \cite{KulliKattimani}). Thus assume that $G$ is a unicyclic graph which is not a cycle. Then $G$ has a leaf, say $v$. Now, if $u$ is the only neighbor of $v$, then $(V_G \setminus \{v\}) \cup \{(v,vu)\}$ is a minimum  dominating set of $S_2(G)$. Since $S_2(G)-((V_G \setminus \{v\})\cup \{(v,vu)\})$ has $|V_G|+1$ components, $(V_G \setminus \{v\})\cup \{(v,vu)\}$ is a minimum accurate dominating set of $S_2(G)$ and $\adom(S_2(G))= |(V_G \setminus \{v\})\cup \{(v,vu)\}|= |V_G|$.

\emph{Case 3.  $|E_G|=|V_G|-1$.} In this case, $G$ is a tree. Now, if $G=K_1$, then $S_2(G) =K_1$ and  $\adom(S_2(G))= \adom(K_1)=1=|V_G|$. If $G=K_2$, then $S_2(G)=P_4$ and  $\adom(S_2(G))= \adom(P_4)=3=2+1=|V_G|+1$. Finally, if $G$ is a tree of order at least three, then the tree $S_2(G)$ is not a~corona graph and by (1) and Theorem~\ref{t:trees} we have $\adom(S_2(G))= \gamma(S_2(G))= |V_G|$. \end{proof}

As a consequence of Theorem~\ref{tw-2-subdivision}, we have the following results.

\begin{cor}
\label{wniosek3}
If $T$ is a tree and $\cP =\{\mathcal{P}(v) \colon v\in V_T\}$ is a family of partitions of the vertex neighborhoods of $T$, then
\[
\adom(T\circ \mathcal{P}) = \left\{
\begin{array}{cl}
|V_T|+1 & \mbox{if $|\cP (v)|=1$ for every $v \in V_T$} \2 \\
|V_T| & \mbox{if $|\cP (v)|>1$ for some $v \in V_T$.} \end{array}
\right.
\]
\end{cor}
\begin{proof} If $|\cP (v)|=1$ for every $v \in V_T$, then $T\circ \mathcal{P}= T\circ K_1$ and the result follows from Corollary \ref{twierdzenie-corona2}. If $|\cP (v)|>1$ for some $v \in V_T$, then the tree $T\circ \mathcal{P}$ is not a~corona and the result follows from Theorem~\ref{t:trees} and Theorem~\ref{tw-general-corona}\,(1).
\end{proof}

\begin{cor} \label{wniosek4}
For $n \ge 3$, if $\cP =\{\mathcal{P}(v) \colon v\in V_{C_n}\}$ is a family of partitions of the vertex neighborhoods of $C_n$, then
\[
\adom(C_n\circ \mathcal{P}) = \left\{
\begin{array}{cl}
n+1 & \mbox{if $|\mathcal{P}(v)|=1$ for every $v\in V_{C_n}$} \2 \\
n+2 & \mbox{if $|\mathcal{P}(v)|=2$ for every $v\in V_{C_n}$} \2 \\
n & \mbox{otherwise.}
\end{array}
\right.
\]
\end{cor}

\noindent \begin{proof} If $|\cP (v)|=1$ for every $v \in V_{C_n}$, then ${C_n}\circ \mathcal{P}= {C_n}\circ K_1$. Thus, by Theorem \ref{twierdzenie-corona2}, we have $\adom(C_n \circ \mathcal{P})= \adom(C_n \circ K_1)= \gamma(C_n \circ K_1)= |V_{C_n}|+1=n+1$.

If $|\cP (v)|>1$  (and therefore $|\cP (v)|=2$) for every $v \in V_{C_n}$,
then ${C_n}\circ \mathcal{P}= S_2(C_n)= C_{3n}$. Now, since $\adom(C_{3n})= n+2$ (as it was observed in \cite{KulliKattimani}), we have $\adom({C_n}\circ \mathcal{P})=\adom(C_{3n})= n+2$.

Finally assume that there are vertices $u$ and $v$ in $C_n$ such that $|\cP (v)|=1$  and $|\cP (u)|=2$. Then the sets
\[
V_{C_n}^1=\{x\in V_{C_n}\colon |\cP (x)|=1\} \hspace*{0.5cm} \mbox{and} \hspace*{0.5cm} V_{C_n}^2=\{y\in V_{C_n}\colon |\cP (y)|=2\}
\]
form a partition of $V_{C_n}$.
Without loss of generality we may assume that $x_1,x_2,\ldots,x_k$, $y_1,y_2,\ldots, y_{\ell}, \ldots$, $z_1, z_2,\ldots, z_p, t_1, t_2,\ldots, t_q$ are the consecutive vertices of $C_n$, where \[
x_1,x_2,\ldots, x_k \in V_{C_n}^1, y_1,y_2,\ldots, y_{\ell}\in V_{C_n}^2, \ldots, z_1, z_2,\ldots, z_p\in V_{C_n}^1,  t_1, t_2,\ldots, t_q\in V_{C_n}^2,
\]
and $k+{\ell}+\ldots+p+q=n$. It is easy to observe that  $D=\{(x_i,N_{C_n}(x_i))\colon i=1,\ldots,k\}\cup \{(y_j,1)\colon j=1,\ldots,{\ell}\}\cup \cdots \cup \{(z_i,N_{C_n}(z_i))\colon i=1,\ldots, p\}\cup \{(t_j,1)\colon j=1,\ldots,q\}$ is a dominating set of ${C_n}\circ \mathcal{P}$. Since the set $D$ is of cardinality $n=|V_{C_n}|$ and $n= \gamma({C_n}\circ \mathcal{P})$ (by Theorem \ref{tw-general-corona}\,(1)), $D$ is a minimum dominating set of ${C_n}\circ \mathcal{P}$. In addition, since ${C_n}\circ \mathcal{P}-D$ has $k+(2+({\ell}-1))+\ldots
+p+(2+(q-1))> k+{\ell}+\ldots+p+q=n$ components, that is, since $\kappa({C_n}\circ \mathcal{P}-D)>n$, no $n$-element subset of $V_{{C_n}\circ \mathcal{P}} \setminus D$ is a dominating set of ${C_n}\circ \mathcal{P}$. Thus, $D$ is an accurate dominating set of  ${C_n}\circ \mathcal{P}$ and therefore $\gamma({C_n}\circ \mathcal{P})=n$. \end{proof}

\section{Closing open problems}

We close with the following list of open problems that we have yet to settle.

\begin{Problem} \rm Find a formula for the accurate domination number  $\adom(G\circ \cF )$ of the $\cF $-corona of a graph $G$ depending only on the family $\cF =\{F_v\colon v\in V_G\}$ such that $\gamma(F_v)=1$ for every $v\in V_G$. \end{Problem}

\begin{Problem} \rm Characterize the graphs $G$ and the families $\cP =\{\cP (v) \colon v\in V_G\}$ for which $\adom(G\circ \mathcal{P})=  |V_G|+\min\{ \min\{|\cP (v)|\colon v\in V_G\}, 1+\min\{|A|\colon A\!\in \!\bigcup_{v\in V_G}\!\cP (v)\}\}$. \end{Problem}

\begin{Problem} \rm It is a natural question to ask if there exists a nonnegative integer $k$ such that $\adom(G\circ \mathcal{P})\le |V_G|+k$ for every graph $G$ and every family $\cP =\{\mathcal{P}(v) \colon v\in V_G\}$ of partitions of the vertex neighborhoods of $G$. \end{Problem}

\end{document}